\documentclass[12pt]{amsart}
\usepackage[matrix,arrow,curve]{xy}
\usepackage{euscript, amssymb, array}
\makeatletter 
\@addtoreset{equation}{subsection} 
\makeatother
\newcommand{\type}[1]{$\mathrm{(#1)}$}

\newcommand{\wt}{\operatorname{wt}}

\newcommand{\muu}{\mbox{\boldmath $\mu$}}
\newcommand{\Cl}{\operatorname{Cl}}

\newcommand{\Clsc}{\operatorname{Cl^{\operatorname{sc}}}}
\newcommand{\gr}{\operatorname{gr}}
\newcommand{\red}{\operatorname{red}}

\newcommand{\q}{\operatorname{q}}
\newcommand{\Supp}{\operatorname{Supp}}
\newcommand{\Sing}{\operatorname{Sing}}

\newcommand{\mt}[1]{\operatorname{#1}}

\newcommand{\OOO}{\mathcal{O}}

\newcommand{\FFF}{\mathcal{F}}
\newcommand{\CC}{\mathbb{C}}

\newcommand{\PP}{\mathbb{P}}
\newcommand{\QQ}{\mathbb{Q}}

\newcommand{\xref}[1]{{\rm \ref{#1}}}

\newcommand{\comment}[1]{}

\newcounter{THN}[section]
\renewcommand{\theTHN}
{(\arabic{section}.\arabic{subsection})}
\newcounter{THNO}[section]
\renewcommand{\theTHNO}
{(\arabic{section}.\arabic{subsection}.\arabic{equation})}

\newenvironment{mparag}[1]{
\setcounter{THN}{\value{subsection}}
\refstepcounter{subsection}\refstepcounter{THN}
\par\medskip\noindent\begingroup \rm
{\bf\theTHN\ #1\ }}{\par\smallskip\endgroup}

\newenvironment{mtparag}[1]{
\setcounter{THN}{\value{subsection}}
\refstepcounter{subsection}\refstepcounter{THN}
\par\medskip\noindent\begingroup \it
{\bf\theTHN\ #1\ }}{\par\smallskip\endgroup}

\newenvironment{parag}[1]{
\setcounter{THNO}{\value{equation}}
\refstepcounter{equation}\refstepcounter{THNO}
\par\medskip\noindent\begingroup \rm
{\bf\theTHNO\ #1\ }}{\par\smallskip\endgroup}

\newenvironment{tparag}[1]{
\setcounter{THNO}{\value{equation}}
\refstepcounter{equation}\refstepcounter{THNO}
\par\medskip\noindent\begingroup \it
{\bf\theTHNO\ #1\ }}{\par\smallskip\endgroup}

\newenvironment{example}{\begin{parag}{Example.}}{\end{parag}}

\newenvironment{lemma}{\begin{tparag}{Lemma.}}{\end{tparag}}
\newenvironment{claim}{\begin{tparag}{Claim.}}{\end{tparag}}

\newenvironment{emptytheorem}{\begin{tparag}{}}{\end{tparag}}
\newenvironment{corollary}{\begin{tparag}{Corollary.}}{\end{tparag}}

\newenvironment{theoremm}{\begin{mtparag}{Theorem.}}{\end{mtparag}}

\newenvironment{propositionm}{\begin{mtparag}{Proposition.}}{\end{mtparag}}
\newenvironment{definitionm}{\begin{mparag}{Definition.}}{\end{mparag}}

\title{On $\QQ$-conic bundles, II}
\author{Shigefumi Mori}
\author{Yuri Prokhorov}
\thanks {
The research of the first author was supported in part by JSPS Grant-in-Aid
for Scientific Research (B)(2), No. 16340004.
The second author was
partially supported by grants CRDF-RUM, No. 1-2692-MO-05 and
RFBR, No. 05-01-00353-a, 06-01-72017.}
\address{Shigefumi Mori: RIMS, 
Kyoto University, Oiwake-cho, Kitashirakawa, Sakyo-ku, Kyoto
606-8502, Japan}
\email{mori@kurims.kyoto-u.ac.jp}
\address{Yuri Prokhorov: Department 
of Algebra, Faculty of Mathematics, Moscow State
University, Moscow 117234, Russia}
\email{prokhoro@mech.math.msu.su}

\begin{document}
\maketitle

\begin{abstract}
A $\mathbb Q$-conic bundle germ is a proper morphism from a threefold
with only terminal singularities to the germ $(Z \ni o)$ of a normal 
surface such that fibers are connected and the anti-canonical 
divisor is relatively ample. 
We obtain the complete classification of $\mathbb Q$-conic
bundle germs when the base surface germ is singular.
This is a generalization of \cite{Mori-Prokhorov-2006},
which further assumed that the fiber over $o$ is irreducible.
\end{abstract}

\section{Introduction}
This note is a continuation of our previous work \cite{Mori-Prokhorov-2006}
where we studied the local structure of $\QQ$-conic bundles.
\begin{definitionm}
A \textit{$\QQ$-conic bundle} is a projective morphism
$f\colon X \to Z$ from a threefold with only terminal singularities to a
surface such that
\begin{enumerate}
\item
$f_*\OOO_X=\OOO_Z$ and all fibers are one-dimensional,
\item
$-K_X$ is $f$-ample.
\end{enumerate}
For $f\colon X\to Z$ as above and for a point $o\in Z$, we call the
analytic germ $(X, f^{-1}(o)_{\red})$ a \textit{$\QQ$-conic
bundle germ}.
\end{definitionm}
In \cite{Mori-Prokhorov-2006} we completely classified $\QQ$-conic bundle germs 
over a singular base and such that the central fiber is irreducible.
For convenience of quotations we reproduce 
briefly  the classification.
For more detailed explanations we refer to the original paper
\cite{Mori-Prokhorov-2006}.

\begin{theoremm}
\label{th-MP-06}
Let $f\colon (X,C)\to (Z,o)$ be a $\QQ$-conic bundle germ, where
$C$ is irreducible and $(Z,o)$ is singular. Then  we are in one of the 
following cases\textup:
\par
\begin{center}
\renewcommand{\extrarowheight}{4pt}
\begin{tabular}{|l|l|p{160pt}|l|}
\hline
\textup{Type}&\textup{No.
}&{\rm singularities}& $(Z,o)$
\\[7pt]
\hline
{\rm toroidal}& \textup{(1.2.1)}&  $\frac1n (1,a,-a)$ \textup{and} $\frac1n (-1,a,-a)$,
$\gcd(n,a)=1$&$A_{n-1}$
\\
\type{IA}\textup+\type{IA}& \textup{(1.2.2)}&
$\frac1n (a,-1,1)$ \textup{and} $\frac1n (a+1,1,-1)$, $n=2a+1$
&$A_{n-1}$
\\
\type{IE^\vee}& \textup{(1.2.3)}&$\frac18(5,1,3)$&$A_{3}$
\\
\type{ID^\vee}& \textup{(1.2.4)}&$cA/2$ \textup{or} $cAx/2$&$A_{1}$
\\
\type{IA^\vee}& \textup{(1.2.5)}&$\frac14(1,1,3)$\quad  \textup(\textup+\type{III}\textup)&$A_{1}$
\\
\type{II^\vee}& \textup{(1.2.6)}&$cAx/4$\quad  \textup(\textup+\type{III}\textup)&$A_{1}$
\\
\hline
\end{tabular}
\end{center}
\end{theoremm}

In this paper we consider the case where the base surface is singular
and the central fiber is reducible.
Our main result is the following.
\begin{theoremm}
\label{th-m}
Let $f\colon (X,C)\to (Z,o)$ be a $\QQ$-conic bundle
germ.
Assume that $C$ 
is reducible and the base surface $(Z,o)$ is singular. 
Then 
$(Z,o)$ is Du Val of type $A_1$ and 
$(X,C)$ is the $\muu_2$-quotient of the index-two $\QQ$-conic bundle 
$f'\colon (X',C')\to (Z',o')$ over a smooth base, where 
$\muu_2$ acts on $(Z',o')$ freely in codimension one.
Moreover, $C'$ has four irreducible components, $\muu_2$ does not fix
any of 
them and $X$ has a unique non-Gorenstein point $P$. 
Furthermore, 
$X'$ is given by the following two equations in
$\PP(1,1,1,2)_{y_1,\dots,y_4}\times \CC^2_{u,v}$
\[
\left\{
\begin{array}{lll}
y_1^2-y_3^2&=&\psi_1(y_1,\dots,y_4;u,v),
\\[7pt]
y_2^2-y_3^2&=&\psi_2(y_1,\dots,y_4;u,v),
\end{array}
\right.
\]
where  $\muu_{2}$ acts as follows\textup:
\[
(y_1, y_2, y_3, y_4; u, v) \longmapsto 
(-y_1, -y_2, y_3, -y_4; -u, -v). 
\]
Here $\psi_i=\psi_i(y_1,\dots, y_4; u, v)$ 
are weighted quadratic in $y_1,\dots,y_4$
with respect to $\wt (y_1,\dots,y_4)=(1,1,1,2)$ and
$\psi_i(y_1,\dots,y_4;0,0)=0$. 
The following are the only possibilities\textup:
\begin{emptytheorem}
\label{item=main--th-cyclic-quo}
$(X,P)$ is a cyclic quotient singularity of type $\frac14(1,1,-1)$
and for any component $C_i\subset C$ germ $(X,C_i)$ is of type \type{IA^\vee},
\end{emptytheorem}

\begin{emptytheorem}
\label{item=main--th-cAx/4}
$(X,P)$ is a singularity of type $cAx/4$
and for any component $C_i\subset C$ germ $(X,C_i)$ is of type \type{II^\vee}.
\end{emptytheorem}
Conversely, if the quotient $(X,C)=(X',C')/\muu_2$, where $(X',C')$
and the action of $\muu_2$ are as above, has 
only terminal singularities, then $(X,C)$ is a conic bundle germ 
over $\CC^2_{u,v}/\muu_2$ with reducible central fiber $C$.
\end{theoremm}
Below are a series of explicit examples of $\QQ$-conic bundles as in 
\ref{th-m}.
\begin{example}
\label{example-4}
Consider
the subvariety $X'\subset\PP(1,1,1,2)\times \CC^2$ defined by
the following two equations:
\[
\left\{
\begin{array}{lll}
y_1^2-y_3^2+u^{2k+1}y_4+v^2y_2 ^2&=&0,
\\[7pt]
y_2^2-y_3^2+vy_4&=&0.
\end{array}
\right.
\]
The projection $f'\colon X'\to\CC^2$ is a $\QQ$-conic bundle of
index $2$ (cf. \cite[12.1.3]{Mori-Prokhorov-2006}). 
Define the action of $\muu_2$ on $X'$ as follows
\[
(y_1,\ y_2,\ y_3,\ y_4;\ u,\ v) \longmapsto
(- y_1,\ - y_2,\ y_3,\ - y_4;\ -u,\ -v).
\]
Then $X'/\muu_{2}\to
\CC^2/\muu_2$ is a $\QQ$-conic bundle
with a unique non-Gorenstein point $P$.
The point $P$ is of type \ref{item=main--th-cyclic-quo}
if $k=0$ and of type \ref{item=main--th-cAx/4} if $k\ge 1$.
\end{example}

The basic idea of the proof is to 
reduce 
the problem 
of classifying $\QQ$-conic bundles $(X,C)$ as in Theorem \ref{th-m}
to the case 
where the central fiber is irreducible by applying the MMP to a
$\QQ$-factorialization $(X^{\q}, C^{\q})$.
Then the resulting $\QQ$-conic bundle $(\bar X, \bar C)$ belongs to
the list 
\ref{th-MP-06}. 
We trace back from $(\bar X, \bar C)$ to $(X,C)$. 
It turns out that in many cases the steps of the MMP 
do not affect the singularities of $(\bar X, \bar C)$.
Here we use some results about divisorial contractions and flips
(see \S \ref{sect-1})
based on \cite{Kollar-Mori-1992} and \cite{Kawamata-1996}.
Then the base change trick allows us to show that $(X,C)$
is a $\muu_2$-quotient of
an index-two conic bundle, see \S \ref{sect2}.

\subsection*{Acknowledgments}
The work was carried out at Research Institute for Mathematical
Sciences (RIMS), Kyoto University. The second author would like to
thank RIMS for invitations to work there in February 2007, for
hospitality and wonderful conditions of work.

\section{Preliminary results on extremal contractions}
\label{sect-1}
\begin{mparag}{}\label{lemma-singul-cA/2}
Let $(E^\sharp,P^\sharp)$ be a Du Val singularity.  (We assume that 
$(E^\sharp,P^\sharp)$ is \textit{singular}).
Assume that $\muu_m$ acts on $E^\sharp$ freely outside $P^\sharp$
and the quotient  $(E,P)=(E^\sharp,P^\sharp)/\muu_m$ is also Du Val.
Then there is a $\muu_m$-equivariant embedding 
$(E^\sharp,P^\sharp)\subset (\CC^3_{x,y,z},0)$ such that $x$, $y$, $z$
and the equation of $E^\sharp$  are semi-invariant.
Let $F^\sharp\subset \CC^3$ be the locus of points at which 
the action of $\muu_m$ is not free.
By our assumption $F^\sharp$ is a curve.
Define the invariant $\varsigma(E^\sharp,P^\sharp,\muu_m)$ 
as the local intersection number 
$(E^\sharp\cdot  F^\sharp)_0$. 
According to \cite[4.10]{Reid-YPG1987} we have only the following cases:
\renewcommand{\extrarowheight}{4pt}
\begin{equation}
\label{eq-table}
\begin{array}{|l|l|l|}
\hline
m\hspace{20pt}\quad&(E^\sharp,P^\sharp)\to (E,P)\hspace{40pt}\quad &
\varsigma(E^\sharp,P^\sharp,\muu_m)\hspace{30pt}\quad
\\[7pt]
\hline
\text{any}& A_{r-1}\to A_{mr-1} & r
\\
4&A_{2r-2}\to D_{2r+1}&2r-1
\\
2&A_{2r-1}\to D_{r+2}& 2
\\
3&D_4\to E_6&2
\\
2&D_{r+1}\to D_{2r}&r
\\
2&E_6\to E_7&3
\\
\hline
\end{array}
\end{equation}

\begin{parag}{}
Let $(W,P)$ be a three-dimensional terminal singularity of index $m>1$ and 
let $E\in |-K_W|$ be a divisor having a Du Val singularity at $P$.
Assume that $(W,P)$  is not a cyclic quotient. 
Let $\pi\colon (W^\sharp, P^\sharp)\to (W,P)$ be the index-one $\muu_m$-cover
and let $(W^\sharp,P^\sharp)=\{\phi=0\}
\subset \CC^4_{x_1,x_2,x_3,x_4}$ 
be a $\muu_m$-equivariant embedding.
Let $E^\sharp:=\pi^{-1}(E)$ and 
$F^\sharp\subset \CC^3$ be the locus of points at which the action of $\muu_m$ is not free.
Since $\pi$ is free in codimension two, $F^\sharp$ is a curve.
Recall that the local intersection number $(W^\sharp\cdot F^\sharp)_0$
is called the \textit{axial multiplicity} of $(W,P)$ \cite[1a.5]{Mori-1988}.
We denote it by $\mt{am}(W,P)$.
By the classification of terminal singularities we may assume that
$F^\sharp$ is the $x_4$-axis, and either 
$\wt(x_1,x_2,x_3,x_4,\phi)\equiv (1,-1,a,0,0)\mod m$, 
or $m=4$ and $\wt(x_1,x_2,x_3,x_4,\phi)\equiv (1,-1,a,2,2)\mod 4$, 
where $\gcd(a,m)=1$.
Since $(E^\sharp,P^\sharp)$ is a Du Val singularity, 
its Zariski tangent space  at 
the origin is three-dimensional. 
Hence there is a $\muu_m$-stable 
hypersurface $H^\sharp\subset \CC^4$
such that $E^\sharp=H^\sharp\cap W^\sharp$ and $H^\sharp$ is smooth.
\end{parag}

\begin{claim}
$F^\sharp\subset H^\sharp$.
\end{claim}
\begin{proof}
Let $\psi$ be the $\muu_m$-semi-invariant equation of  $H^\sharp$.
Then  $\wt \psi \equiv a$. 
Hence $\psi$ does not contain terms $x_4^k$
and so it vanishes on $F^\sharp$.
\end{proof}

\begin{parag}{}
We define  the invariant $\varsigma(W,E,P)$ as the local intersection number 
$(E^\sharp\cdot  F^\sharp)_0$ inside $H^\sharp$.
Clearly it coincides with $\varsigma(E^\sharp,P^\sharp,\muu_m)$
defined above.
\end{parag}

\begin{lemma}
Assume that $(W,P)$ is not a cyclic quotient singularity.
The invariant $\varsigma(W,E,P)$ does not depend on the choice of $E$ 
and $\varsigma(W,E,P)= \mt{am}(W,P)$.
\end{lemma}
\begin{proof}
Both sides of the equality coincide with
the order of vanishing of $\phi|_{F\sharp}$. 
\end{proof}

\begin{corollary}
\label{cor-lemma-singul-cA/2}
Let $(W,P)$ is a three-dimensional terminal singularity of index $m>1$
which is not a cyclic quotient 
and let $E\in |-K_{(W,P)}|$ be a member having a Du Val singularity of $A$-type at $P$.
Then $E$ is isomorphic to a general member $E_{\mt{gen}}\in |-K_{(W,P)}|$.
\end{corollary}
\begin{proof}
By the above lemma we have $\varsigma(E^\sharp,P^\sharp,\muu_m)=\varsigma(E^\sharp_{\mt{gen}},P^\sharp,\muu_m)
=\mt{am}(W,P)$. Then the statement follows by the first line in \eqref{eq-table}.
\end{proof}
\end{mparag}

\begin{propositionm}
\label{prop-div-contr}
Let $\varphi\colon (V,\Gamma) \to (W,o)$ be 
the analytic germ of a divisorial extremal contraction
of threefolds with terminal singularities 
\textup(in particular, $W$ is $\QQ$-Gorenstein\textup) such that 
the central fiber
$\Gamma:=\varphi^{-1}(o)_{\red}$ is one-dimensional and
irreducible.
\begin{enumerate}
\item
The point $(W,o)$ cannot be of type $cAx/4$.
\item
If $(W,o)$ is of type $cAx/2$, then $(V,\Gamma)$ has a unique non-Gorenstein point 
which is of type \type{II^\vee}.
\item
If $(W,o)$ is analytically isomorphic to 
\begin{equation}
\label{eq-singul-cA/2}
\{ x_1x_2+x_3^2+x_4^{2k}=0 \} / \muu_2(1,1,0,1),
\end{equation}
then $(V,\Gamma)$ has a unique non-Gorenstein point $P$
which is locally imprimitive of index $4$ and splitting degree $2$.
Moreover, $P\in (V,\Gamma)$ is either of type 
\type{II^\vee}
or \type{IA^\vee} and in the second case $(X,P)$ is a cyclic quotient singularity.
\end{enumerate}
\end{propositionm}

\begin{proof}
For the proof we assume that $(W,o)$ is of type $cAx/4$, $cAx/2$,
or as in \eqref{eq-singul-cA/2}.
We will use the classification \cite[Th. 2.2]{Kollar-Mori-1992}. 
Let $m$ be the index of $(W,o)$.
Then the the canonical class $K_W$ 
is an $m$-torsion element in $\Clsc(W,o)$.
Its pull-back $\varphi^*K_W$ is a well-defined Cartier divisor on 
$V\setminus \Gamma$ such that $m(\varphi^*K_W)\sim 0$.
Hence the group $\Clsc(V,\Gamma)$
contains an $m$-torsion element, say $\xi$. 
By the classification \cite[Th. 2.2]{Kollar-Mori-1992}
$\Clsc(V,\Gamma)$ can contain a torsion only
when $(V,\Gamma)$ is of type \type{k1A} (with a point
of type \type{IA^\vee}), \type{II^\vee},
or \type{k2A}.

Assume that $(V,\Gamma)$ is of type \type{k2A}.
Then by \cite[Th. 2.2]{Kollar-Mori-1992} a general 
member $D\in |-K_V|$ and its image
$\varphi(D)\in |-K_W|$ have only Du Val singularities. 
Moreover, $(\varphi(D),o)$ is a singularity of type $A_*$
and so $(W,o)$ is of type $cA/*$. Clearly,
the contraction $\varphi|_D\colon D\to \varphi(D)$
is crepant. By our assumptions 
$(W,o)$ is a singularity given by \eqref{eq-singul-cA/2}.
So, $\mt{am}(W,o)=2$.
By Corollary \ref{cor-lemma-singul-cA/2} the singularity $(\varphi(D),o)$ is 
of type $A_3$. Since $\varphi_D\colon D\to \varphi(D)$ is crepant
and $V$ has two singular points, the only possibility is that 
$D$ has two singularities of type $A_1$.
But in this case $V$ is of index two and 
then by \cite[Th. 4.7]{Kollar-Mori-1992} $V$ has 
a unique non-Gorenstein point, a contradiction. 

In the remaining cases \type{II^\vee} and \type{k1A},
$V$ has a unique non-Gorenstein point $P$. 
Then $(V,\Gamma)$ is locally imprimitive at $P$ and
the splitting degree equals $m$. In particular, the index of $P$
is $>m$ \cite[Cor. 1.16]{Mori-1988}. Thus if $(V,\Gamma)$ is of type \type{II^\vee}, 
then we are in the case (ii) or (iii).

Assume that $(V,\Gamma)$ is of type \type{k1A}.
Then by \cite[Th. 2.2]{Kollar-Mori-1992} a general 
member $D\in |-K_V|$ does not contain $\Gamma$,
has only Du Val singularity at $P:=\{D\cap \Gamma\}$,
and $\varphi|_D\colon D\to \varphi(D)$ is an isomorphism.
Hence $\varphi(D)\in |-K_W|$ has a Du Val singularity of type 
$A$ at $o$. In this case,
$(W,o)$ cannot be of type $cAx/*$. 
Thus $(W,o)$ is given by \eqref{eq-singul-cA/2}.
By Corollary \ref{cor-lemma-singul-cA/2} $D\simeq \varphi(D)$
is of type $A_3$. Since the index of $(V,P)$ is $>2$,
$(V,P)$ must be a cyclic quotient singularity 
$\frac14(1,1,-1)$. So we are in the case (iii). This proves the proposition.
\end{proof}

\begin{propositionm}
\label{lemma-flip-1}
Let $\chi\colon (V,\Gamma) \dashrightarrow (V^+,\Gamma^+)$
be a flip of threefolds with terminal singularities
with irreducible flipping curve $\Gamma$. 
Then $(V^+,\Gamma^+)$ 
contains none of the following configurations of singularities:
\begin{enumerate}
\item
two cyclic quotient singularities $P_1^+$ and $P_2^+$ of indices $m_1$ and $m_2$
with $\gcd(m_1,m_2)>1$ such that $(V^+,\Gamma^+)$ is locally 
primitive at $P_1^+$ and $P_2^+$;
\item
an imprimitive point $P^+$
of splitting degree $s>1$.
\end{enumerate}
\end{propositionm}

\begin{proof}
By \cite[Cor. 13.4]{Kollar-Mori-1992} $\Gamma^+$ is irreducible.
Assume that one of the cases (i)-(ii) holds.
As in \cite[Cor. 1.12]{Mori-1988} there is a $d$-torsion 
element $\xi^+\in \Clsc V^+$ for some $d>1$. Its proper transform $\xi$ on $V$
is a $d$-torsion element in $\Clsc V$.
In \cite{Kollar-Mori-1992} flips are classified 
into 6 types \type{k1A}, \type{k2A}, \type{cD/3}, \type{IIA}, \type{IC}, \type{kAD}
according to a general member 
of the anticanonical linear system $|-K_{V}|$ 
\cite[Th. 2.2]{Kollar-Mori-1992}. 
The group $\Clsc V$ can contain a torsion only in cases \type{k1A} and \type{k2A}
(in all other cases the flipping variety is locally primitive and 
indices of non-Gorenstein points are coprime, 
cf. \cite[Cor. 1.12]{Mori-1988}). 
The torsion elements $\xi$ and $\xi^+$ induce 
the following cyclic $\muu_d$-coverings:
\begin{equation}
\label{eq-diag-flips}
\xymatrix{
(V',\Gamma')\ar@{-->}[r]^{\chi'}\ar[d]^{\pi}&
(V^{+\prime},\Gamma^{+\prime})\ar[d]^{\pi^+}
\\
(V,\Gamma)\ar@{-->}[r]^{\chi}&(V^{+},\Gamma^+)
}
\end{equation}

Consider the flipping diagram
\[
\xymatrix{ (V,\Gamma) \ar@{-->}[rr]^{\chi}\ar[dr]_{\varphi}
&&(V^+,\Gamma^+)\ar[dl]^{\varphi^+}
\\
&(W,o)&
}
\]
By \cite[Th. 7.3, 9.10]{Mori-1988} and
\cite[Th. 2.2]{Kollar-Mori-1992}, a general member $D\in |-K_{V}|$
has only Du Val singularities. Since the restriction
$\varphi_D\colon D\to \varphi(D)$ is crepant, the same holds 
for $\varphi(D)\in |-K_W|$. Further, if we put $D^+=\chi(D)$,
then $D^+\in |-K_{V^+}|$ and $D^+$ also has 
only Du Val singularities. 
Since $K_{V^+}\cdot \Gamma^+>0$, $D^+\supset \Gamma^+$.

\begin{parag}{}
First we consider the case where our flip is of type \type{k1A}.
Then $V$ has a unique non-Gorenstein point $P$
and $P$ is of type $cA/*$.
In this case $D\cap \Gamma =\{P\}$ and $(\varphi(D),o)\simeq (D,P)$ 
is of type $A_*$. 
Since $\Clsc V$ has a torsion, $(V,\Gamma)$ is locally imprimitive at $P$. 
\end{parag}

\begin{parag}{}
Assume that we are in the case (i).
We claim that $V^+$ has at least one Gorenstein singular point.
Indeed, 
since the germ 
$(V,\Gamma)$ has only one non-Gorenstein point, it is locally imprimitive
and in the diagram \eqref{eq-diag-flips} 
$\pi$ is the splitting cover \cite[Cor. 1.12]{Mori-1988}. 
Here $\Gamma'$ has exactly $d$ components
and $V ^{+\prime}$ is the relative canonical model of $V'$. 
Since $(V^{+},\Gamma^+)$ is locally primitive at $P_1^+$ and $P_2^+$,
the curve $\Gamma^{+\prime}$ is irreducible.
Now the map $\chi'$ can be decomposed as follows
\[
\chi'\colon V'=V_0' \dashrightarrow V_1' \dashrightarrow 
\cdots \dashrightarrow V_n' \to V ^{+\prime},
\]
where every $V_i' \dashrightarrow V_{i+1}'$ is a flip along an irreducible curve
and $V_n' \to V ^{+\prime}$ is a crepant small contraction
(cf. \cite[Proof of 13.5]{Kollar-Mori-1992}).
Every step $V_i' \dashrightarrow V_{i+1}'$ preserves the number 
of components of the central fiber. Hence the crepant 
contraction $V_n' \to V ^{+\prime}$ is nontrivial and gives 
us a Gorenstein non-$\QQ$-factorial point $Q\in \Gamma^+\subset V^{+}$. 
This proves our claim. Thus the divisor $D^+$ has at least three 
singular points: $P_1^+$, $P_2^+$, and $Q$.
But then $\varphi^+_D\colon D^+\to \varphi(D)$ 
contracts $\Gamma^+$ to a Du Val singularity of type 
$D_*$ or $E_*$, a contradiction. 
\end{parag}

\begin{parag}{}
\label{par-plt-1}
Now we assume that we are in the case (ii). 
We claim that the log divisor $K_{D^+}+\Gamma^+$ is not plt at $P^+$.
Indeed, in the diagram \eqref{eq-diag-flips} $\pi^+$ is the splitting cover 
(see \cite[Cor. 1.12.1]{Mori-1988}).
In particular, $\pi^+$ is \'etale outside $P^+$,
$\pi^{+-1}(P^+)$ is one point, and $\Gamma^{+\prime}$ has 
$s>1$ irreducible components, all of them pass through $\pi^{+-1}(P^+)$. 
Let $D^{+\prime}:=\pi^{+-1}(D^+)$. 
Since $\Gamma^{+\prime}$ is singular at $\pi^{+-1}(P^+)$,
the log divisor $K_{D^{+\prime}}+\Gamma^{+\prime}$ is not plt at this point.
This proves our claim because the restriction $\pi^+_D\colon D^{+\prime}
\to D^+$ is \'etale in codimension one (see, e.g., \cite[Cor. 20.4]{Utah}).
Now since the contraction $D^+\to \varphi (D)$ is crepant,
$D^+$ is dominated by the minimal resolution $D^{\min}$ of $\varphi(D)$:
$D^{\min}\to D^+\to \varphi (D)$.
Since $K_{D^+}+\Gamma^+$ is not plt, the exceptional divisor of 
$D^{\min}\to \varphi(D)$ is not a chain of smooth rational curves.
Hence $(\varphi(D),o)$ is not a singularity of type $A_*$, a contradiction.
\end{parag}

\begin{parag}{}
Finally, we consider the case where our flip is of type \type{k2A}
These flips are described in \cite{Mori-2002}.
We will use notation of \cite{Mori-2002}.
By \cite[Th. 4.7]{Mori-2002} 
$(V^+,\Gamma^+)$ is locally primitive.
Hence we have the case (i). Moreover,
$V^+$ has exactly two singular points and they 
are analytically isomorphic to germs of the following $cA/m_i$ 
singularities:
\[
\{\xi_i\eta_i=G_{k-i}(\zeta_i^{m_i}, u^{e(k+2-i)})\}/\muu_{m_i}
\subset \CC_{\xi_i,\eta_i,\zeta_i,u}^4
/\muu_{m_i}(1,-1,a_i,0),
\]
where $k$, $a_i$ are some positive numbers and $e(j)$ is some function.
Hence these 
points coinside with $P_1^+$ and $P_2^+$.
Since $P_i^+\in \Gamma^+\subset V^+$ are cyclic quotient singularities,
we have $e(k)=e(k+1)=1$ ($u$ needs to be eliminated).
If we put $\delta:=a_1m_2+a_2m_1-m_1m_2$, then $\delta\ge d$ and
by definition \cite[Def. 3.2]{Mori-2002} we have $e(3)=0$, 
$e(4)=\delta\alpha_1\ge d>1$,
$e(5)=(\delta^2\rho_2-1)\alpha_1+\delta\alpha_2\ge d>1$
(see \cite[Rem. 3.6]{Mori-2002}).
Thus, $k\ge 6$.
On the other hand, by \cite[Lemma 3.5, Cor. 3.7]{Mori-2002}
we have $k\le 5$, a contradiction.
\end{parag}
\end{proof}

\begin{propositionm}
\label{prop-crepant-contr}
Let $\varphi\colon (V,\Gamma)\to (W,o)$ be the germ of a birational crepant
contraction of threefolds with terminal singularities, where
$\Gamma$ is irreducible.
\begin{enumerate}
\item 
$(V,\Gamma)$ contains at most two non-Gorenstein points. 
\item
If $(V,\Gamma)$ is imprimitive at some point $P$, then
$ (W,o)$ cannot be a singularity of type $cA/*$.
\end{enumerate}
\end{propositionm}

\begin{proof}
For the proof we assume that $V$ is not Gorenstein.
Since $\varphi$ is crepant, the point $ (W,o)$
is not Gorenstein. Let $m$ be its index.
Let $D\in |-K_{(W,o)}|$ be a general member and let
$S:=\varphi^{-1}(D)$. Then $S\in |-K_{ (V,\Gamma)}|$
and both $S$ and $D$ have only Du Val singularities.
Moreover, the restriction map $\varphi_S\colon S \to D$ is crepant.
Hence $S$ is dominated by the minimal resolution 
$D^{\min}$ of $D$ and obtained from $D^{\min}$ by contracting 
all but one exceptional curves. 

First assume that $(V,\Gamma)$ has at least three non-Gorenstein points, 
say $P$, $Q$, and $R$. 
By the classification of Du Val singularities $(D,o)$ is a singularity of type $D_*$ or $E_*$
and $S$ is obtained from $D$ by blowing up the 
exceptional curve corresponding to the 
central vertex in the Dynkin diagram.
In this case exceptional curves on $D^{\min}$ over $(S,P)$, $(S,Q)$ and $(S,R)$
form strings and the proper transform of $\Gamma$ 
is adjacent to the ends of them.
This means that the log divisor $K_{S}+ \Gamma$ is plt. 
The latter implies that the germ $(V, \Gamma)$
is locally primitive (cf. \ref{par-plt-1}). 
Now consider the
index-one cover $\pi \colon (W^\sharp, o^\sharp)\to (W,o)$.
It induces the following diagram
\begin{equation}
\label{eq-diag-1}
\xymatrix{
 (V^{\sharp}, \Gamma^{\sharp})\ar[r]^{\upsilon}\ar[d]^{\varphi^\sharp}
& (V, \Gamma)\ar[d]^{\varphi}
\\
 (W^\sharp, o^\sharp)\ar[r]^\pi& (W,o)
}
\end{equation}
Since $(V, \Gamma)$
is locally primitive, $\Gamma^{\sharp}=\pi^{\sharp}(o^\sharp)$
is irreducible. The group $\muu_{m}$ naturally acts on $\Gamma^{\sharp}
\simeq \PP^1$ and has exactly two fixed points.
Thus we may assume that $\upsilon^{-1}(R)$
contains no fixed points.
But then $\upsilon^{-1}(R)$ consists of $m>1$
non-Gorenstein points of the same index.
By \cite[Cor. 1.12]{Mori-1988} there is a torsion element 
in $\Clsc (V^{\sharp}, \Gamma^{\sharp})\simeq \Clsc (W^\sharp, o^\sharp)$.
This contradicts the fact that $W^\sharp\setminus \{ o^\sharp\}$ is simply connected.
Thus (i) is proved.

Now assume that $(V,\Gamma)$ contains an imprimitive point $P$.
By the proof of (i) $S$ has at most two singular points
and the log divisor $K_S+\Gamma$ is not plt at $P$.
On the other hand, assume that $(D,o)$ is a point of type $A_*$.
Then the exceptional curves of 
the minimal resolution $D^{\min}\to S$ and $\Gamma$ form a 
chain. Hence $K_S+\Gamma$ is not plt, a contradiction.
\end{proof}

\begin{mtparag}{Proposition (cf. \cite[1.14]{Mori-1988}).}
\label{claim-non-Gorenstein}
Let $f\colon (X,C)\to (Z,o)$ be the germ of a contraction from a threefold with only
terminal singularities to a surface such that 
\begin{enumerate}
\item
$-K_X$ is nef and big,
\item
$C:=f^{-1}(o)_{\red}$ is a curve having 
at least three components,
\item
each $K_X$-trivial component $C_j\subset C$ contains a non-Gorenstein point.
\end{enumerate}
Then $X$ has index $>1$
at all singular points of $C$.
\end{mtparag}
\begin{proof}
By the Kawamata-Viehweg vanishing theorem we have 
$R^1f_*\OOO_X=0$. Hence $C$ is a union of $\PP^1$'s 
whose configuration is a tree.
Let $P\in C$ be a singular point and let $C_i\subset C$
be a component passing through $P$.
We have $\gr_{C_i}^0\omega\simeq \OOO(-1)$.
Indeed, take a positive integer $m$ such that $mK_{X}$
is Cartier. Then there is a natural embedding 
$(\gr_{C_i}^0\omega)^{\otimes m} \hookrightarrow 
\OOO_{C_i}(mK_{X})$.
Since $K_{X}\cdot C_i\le 0$ we have 
$\deg \gr_{C_i}^0\omega\le 0$. 
Moreover, if $K_{X}\cdot C_i< 0$, then 
$\deg \gr_{C_i}^0\omega< 0$. Assume that $K_{X}\cdot C_i= 0$
Since $C_i$ contains a non-Gorenstein point, the above embedding is 
not an isomorphism and so again $\deg \gr_{C_i}^0\omega< 0$. 
On the other hand, $C_i$ is contractible over $Z$. Hence,
by the Grauert-Riemenshneider vanishing theorem 
we have $H^1(\gr_{C_i}^0\omega)=0$.
This shows $\gr_{C_i}^0\omega\simeq \OOO(-1)$.

Now let $C_j$ be another component of $C$ passing through $P$.
As above, $\gr_{C_j}^0\omega\simeq \OOO(-1)$.
Consider the following exact sequence
\[
0 \longrightarrow \gr_{C_i\cup C_j}^0\omega 
\longrightarrow
\gr_{C_i}^0\omega \oplus \gr_{C_j}^0\omega
\longrightarrow \FFF \longrightarrow 0,
\]
where $\Supp \FFF =P$.
Since $C_i \cup C_j\neq C$, $C_i\cup C_j$ is contractible over $Z$
and again by the Grauert-Riemenshneider vanishing 
$H^1(\gr_{C_i\cup C_j}^0\omega)=0$.
This implies $\gr_{C_i\cup C_j}^0\omega 
\simeq
\gr_{C_i}^0\omega \oplus \gr_{C_j}^0\omega$.
So $\gr_{C_i\cup C_j}^0\omega $ is not locally free 
at $P$ and this point cannot be Gorenstein.
\end{proof}

\section{The proof of the main theorem}
\label{sect2}
In this section we prove Theorem \ref{th-m}.
\begin{mparag}{Notation.}
\label{m-not}
Let $f\colon (X,C)\to (Z,o)$ be a $\QQ$-conic bundle
germ with reducible central fiber $C$. Then $\rho(X/Z)>1$. 
Recall that according to \cite[Th. 1.2.7]{Mori-Prokhorov-2006}
$(Z,o)$ is either smooth or Du Val of type $A$.
We assume that $(Z,o)$ is singular of type $A_{n-1}$, $n\ge 2$.

\begin{lemma}
\label{lemma-start}
Notation as above.
\begin{enumerate}
\item
If $(X,C)$ has a point $P$  such that either 
\begin{enumerate}
\item
$P$ is of type $cAx/4$, or
\item
for each component $C_i\subset C$ passing through $P$
the germ $(X,C_i)$ is locally
imprimitive at $P$.
\end{enumerate}
Then $P$ is the only non-Gorenstein point on $X$.
\item
Conversely, if $P$ is a unique non-Gorenstein point on $X$,
then all the components $C_i\subset C$ pass through $P$
and the germ $(X,C_i)$ is locally
imprimitive at $P$. 
If furthermore $(X,P)$ is of index $4$, then 
$(X,C)$ is a quotient of an index two $\QQ$-conic bundle germ $(X',C')$
over a smooth base by $\muu_2$, where the action is free in codimension one,
$C'$ has four irreducible components and $\muu_2$ does not fix 
any of them.
\end{enumerate}
\end{lemma}

\begin{proof}
Let $P\in X$ be a point as in (i). 
For each component $C_i\subset C$ passing through $P$
the germ $(X,C_i)$ is an extremal neighborhood and by 
\cite[Th. 2.2]{Kollar-Mori-1992} $(X,C_i)$ has no non-Gorenstein point
other than $P$. Since each singular point of $C$ is not Gorenstein 
\cite[Prop. 4.2]{Kollar-1999-R}, 
\cite[4.4.2]{Mori-Prokhorov-2006} and 
$C$ is connected, $P$ is the only non-Gorenstein point on the whole $X$.

Now assume that $P$ is the only non-Gorenstein point.
Consider the base change \cite[2.4]{Mori-Prokhorov-2006}:
$(X',C')\to (X,C)$. Here $(X',C')$ is a conic bundle germ over a smooth base
and $X'\to X$ is an \'etale outside $P$ $\muu_n$-cover.
Thus $(X,C)=(X',C')/\muu_n$.
If $\muu_n$ fixes a component $C_i'\subset C$, then
there are two $\muu_n$-fixed points on $C_i$ and they 
give us two non-Gorenstein points on $X$,
a contradiction.  So the first assertion of (ii) is proved.

Finally assume that $(X,P)$ is of index $4$.
Since the index of $(X,P)$ is divisible by $n$, $n=4$ or $2$..
If $n=4$, then $X'$ is Gorenstein. 
In this case, by \cite[Th. 2.4]{Prokhorov-1997_e} $C$ is irreducible,
a contradiction. Thus $n=2$ and $(X',C')$ is of index $2$.
By the above, 
$\muu_2$ does not fix any component of $C'$.
On the other hand, $C'$ has at most four components
\cite[Th. 12.1]{Mori-Prokhorov-2006}.
Hence $C'$ has exactly four components.
This proves the lemma.
\end{proof}

\begin{parag}{}
Let $q\colon X^{\q}\to X$ be a $\QQ$-factorialization.
(It is possible that $q$ is the identity map.)
Run the MMP over $Z$:
$X^{\q}=X_0 \dashrightarrow X_{N+1}=\bar X$.
Since $X/Z$ is a rational curve fibration, 
$X_{N+1}$ is not a minimal model over $Z$. Therefore, 
at the end we get an extremal contraction 
$\bar f\colon \bar X\to \bar Z$ of Fano type over $Z$.
Since the composition $f^{\q}\colon X^{\q}\to Z$ has only one-dimensional
fibers, 
$Z=\bar Z$ and 
$X^{\q}\dashrightarrow \bar X$ is a sequence of flips and 
extremal 
divisorial contractions that contract a divisor to a curve
which is not contained in the fiber over $o\in Z$.
Thus we have the following diagram:
\[
\xymatrix{
(X^{\q},C^{\q})\ar@/^0.6pc/[ddrr]^{f^{\q}}
\ar@{-->}[r]^{g_{0}}\ar[d]_{q}&(X_1,C_1)\ar@{-->}[r]&
\cdots\ar@{-->}[r]&(X_{N},C_{N})\ar[r]^{g_{N}}
&(\bar X,\bar C)\ar@/^1.4pc/[ddll]_{\bar f} 
\\
(X,C)\ar@/_0.6pc/[drr]^{f}&&&&&
\\
&&(Z,o)&&
}
\]
Here each $X_k$ has a morphism $f_k\colon X_k\to Z$ with 
connected one-dimensional fibers
and $C_k:=f_k^{-1}(o)$
is the central fiber (with reduced structure). 
Since $\rho(\bar X/Z)=1$, 
$\bar f\colon \bar X\to \bar Z$ is a $\QQ$-conic bundle with
irreducible central fiber $\bar C$. 
Since the base $(Z,o)$ is singular, $\bar X$ is not Gorenstein.
So $\bar f$ is classified in \cite{Mori-Prokhorov-2006},
see also \ref{th-MP-06}.
\end{parag}

\begin{parag}{}
Note that each component of the central fiber $C_k$ 
is contractible and the resulting variety is again 
projective over $Z$
(because it has one-dimensional fibers over $Z$). 
Hence each component of $C_k$ 
generates an extremal ray (not necessarily $K$-negative).
This implies that all our flipping curves are irreducible 
and all the divisorial contractions have 
irreducible fibers. 
Note also that all the varieties $X_k$ are analytically 
$\QQ$-factorial at each point on $C_k$ (again because 
$X_k\to Z$ has one-dimensional fibers, cf. \cite[Proof of 1.7]{Mori-1988}).
\end{parag}
\end{mparag}

The following is the key argument in the proof.

\begin{propositionm}
\label{prop-m}
In the above notation one of the following holds.
\begin{tparag}{}
\label{prop-m-2}
There is a component $C_0^{\q}\subset C^{\q}$ containing two cyclic 
quotient singularities $P^{\q}$ and $Q^{\q}$ of index $n$.
No other components of $C^{\q}$ pass through $P^{\q}$ and $Q^{\q}$.
\end{tparag}

\begin{tparag}{}
\label{prop-m-1-1}
There is a point $P^{\q}\in (X^{\q},C^{\q})$ of index $m>1$ 
which is contained in only one 
component $C_0^{\q}\subset C^{\q}$ and 
such that $(X^{\q},C^{\q}_0)$ is locally imprimitive at $P^{\q}$.
The following are the possibilities for $(n,m)$\textup:  $(4,8)$, $(2,4)$, 
and $(2,2)$.
\end{tparag}

\begin{tparag}{}
\label{prop-m-1-2}
There is a point $P^{\q}\in (X^{\q},C^{\q})$ 
which is contained
in exactly two 
components $C_0^{\q}, C_1^{\q}\subset C^{\q}$ and 
such that both germs $(X^{\q},C^{\q}_i)$ are locally imprimitive at $P^{\q}$.
The point $(X^{\q},P^{\q})$ is of type 
$cAx/4$ or $\frac14(1,1,-1)$.  Here $n=2$.
\end{tparag}
Moreover, there is an $n$-torsion element $\xi^{\q}\in \Clsc(X^{\q},C^{\q})$
which is not Cartier at $P^{\q}$ 
\textup(and at $Q^{\q}$ is the case \xref{prop-m-2}\textup).
\end{propositionm}

\begin{proof}
Since $(Z,o)$ is of type $A_{n-1}$, there is an $n$-torsion element 
$\eta\in \Cl(Z,o)$. Put $\bar \xi := \bar f^* \eta$, $\xi_l := f_l^* \eta$, and 
$\xi^{\q} := f^{\q *} \eta$.

Assume that $(\bar X,\bar C)$ is either toroidal of of type 
\type{IA}+\type{IA}.
Let $\bar P,\, \bar Q$ be the singular points of $\bar X$.
Then $\bar \xi$ is not Cartier at $\bar P$ and $\bar Q$.
We claim that the map $\psi \colon \bar X\dashrightarrow X^{\q}$
is an isomorphism near $\bar P$ and $\bar Q$.
Indeed, by induction, 
since $\bar P,\, \bar Q$ are cyclic quotient singularities
of index $n$,
there is no divisorial contractions over these points by \cite{Kawamata-1996}
and by Proposition \ref{lemma-flip-1}
on each step the proper transform of $\bar C$ 
cannot be a flipped curve. So if we put $P^{\q}:=\psi(\bar P)$,
$Q^{\q}:=\psi(\bar Q)$, and $C^q_0:=\psi(\bar C)$, 
we get the case \ref{prop-m-2}.

Now assume that $(\bar X,\bar C)$ is of type 
\type{IE^\vee}, \type{IA^\vee}, or \type{II^\vee}.
Let $\bar P$ be a (unique) non-Gorenstein point. 
Then $(\bar X,\bar P)$ is either a cyclic quotient singularity or
of type $cAx/4$ and again $\bar \xi$ is not Cartier at $\bar P$.
Moreover, $(\bar X,\bar C)$ is locally imprimitive at $\bar P$.
As above, there is no divisorial contractions over $\bar P$
 by \cite{Kawamata-1996} and Proposition \ref{prop-div-contr}
and the proper transform of $\bar C$ 
cannot be a flipped curve 
by Proposition \ref{lemma-flip-1}.
Put $P^{\q}:=\psi(\bar P)$ and $C^q_0:=\psi(\bar C)$.
We get the case \ref{prop-m-1-1}.

Finally consider the case where $(\bar X,\bar C)$ is of type \type{ID^\vee}.
Then $n=2$, i.e., $(Z,o)$ is of type $A_1$.
Let $\bar P$ be a (unique) non-Gorenstein point. 
Then $(\bar X,\bar C)$ is locally imprimitive at $\bar P$ and
$(\bar X,\bar P)$ is 
of type $cA/2$ or $cAx/2$. 
Moreover, in the first case, $(\bar X,\bar P)$ is analytically 
isomorphic to a singularity given by \eqref{eq-singul-cA/2}.
If there is no divisorial contractions over $\bar P$, 
we can argue as above and get the case \ref{prop-m-1-1}.
Otherwise 
on some step, the map 
$\psi_{k+1}\colon \bar X \dashrightarrow X_{k+1}$
is an isomorphism near $\bar P$ and 
there is a divisorial contraction
$g_k\colon X_{k}\to X_{k+1}$ which blows up a curve passsing through 
$P_{k+1}:=\psi_{k+1}(\bar P)$.
Let $C_{k,0}:=g^{-1}_k(P_{k+1})$ and let $C_{k,1}$ be the proper transform of
$\bar C$ on $X_k$. By Proposition \ref{prop-div-contr}
$X_k$ has exactly one non-Gorenstein point $P_k$ on $C_{k,0}$. Moreover, 
$P_k$ is either a cyclic quotient singularity $\frac14(1,1,-1)$ or 
of type $cAx/4$ and $(X_k,C_{k,0})$ is locally imprimitive at 
$P_k$ of splitting degree $2$. 
Note that $\xi_k=g_k^*\xi_{k+1}$ is non-Cartier at all points of $C_{k,0}$.
Since $P_k$ is the only non-Gorenstein point on $C_{k,0}$,
$\xi_k$ is not Cartier at $P_k$.
Now if $C_{k,1}$ does not pass through $P_k$, then as above 
we get the case \ref{prop-m-1-1}. Assume that $C_{k,0}\cap C_{k,1}=\{P_k\}$.

We claim that $(X_k,C_{k,1})$ is locally imprimitive at 
$P_k$. Indeed, $\xi_k$ defines the double cover $\pi_k\colon (X'_k, C'_k)\to (X_k, C_k)$
which is \'etale outside $\Sing X_k$.
Since $\xi_k$ is not Cartier at $P_k$, $\pi_k$ does not split over $P_k$.
Hence, $C_{k,1}':=\pi_k^{-1}(C_{k,1})$ is connected.
On the other hand, 
since $(\bar X, \bar C)$ is locally imprimitive at $\bar P$, the curve $C_{k,1}'$
is reducible.
This means that $C_{k,1}$ is locally imprimitive at $P_k$.
Finally as above the map $X_k\dashrightarrow X^{\q}$
is an isomorphism near $P_k$. We get case \ref{prop-m-1-2}.
\end{proof}

\begin{propositionm}
\label{prop-last}
Notation as in \xref{m-not}.
Then $(X,C)$ contains only one non-Gorenstein point $P$. 
This point is either a cyclic quotient $\frac14(1,1,-1)$
or of type $cAx/4$. Moreover, 
for each component $C_i\subset C$ the germ $(X,C_i)$
is imprimitive at $P$ and 
$(Z,o)$ is of type $A_1$.
\end{propositionm}

\begin{proof}
By Proposition \ref{prop-m} there is a component $C^{\q}_0\subsetneq C^{\q}$
as in \ref{prop-m-2}, \ref{prop-m-1-1}, or \ref{prop-m-1-2}.
First assume that 
$C^{\q}_0$ is not contracted by $q\colon X^{\q}\to X$.
Put $C_0:=q(C^{\q}_0)$. Then $(X,C_0)$ is an extremal neighborhood.
In the case \ref{prop-m-2} it 
has two cyclic quotient singularities at $q(P^{\q})$ and $q(Q^{\q})$
and no other components of $C$ pass through $q(P^{\q})$ and $q(Q^{\q})$.
On the other hand, $C\neq C_0$ and 
intersection points $C_0\cap (C- C_0)$
are non-Gorenstein \cite[Prop. 4.2]{Kollar-1999-R}, 
\cite[4.4.2]{Mori-Prokhorov-2006}. 
Thus the extremal neighborhood $(X,C_0)$ has at least three non-Gorenstein 
points. This contradicts 
\cite[Th. 6.2]{Mori-1988}.
Similarly, in the case \ref{prop-m-1-1},
$(X,C_0)$ is locally imprimitive at $q(P^{\q})$
and no other components of $C$ pass through $q(P^{\q})$.
We get a contradiction by Lemma \ref{lemma-start}.
Consider the case \ref{prop-m-1-2}.
If $C_1^{\q}$ is not contracted by $q$, 
then we are done by  Lemma \ref{lemma-start}.
If $C_1^{\q}$ is contracted by $q$, then 
$q(C_1)$ is a point of type $cAx/4$
by Proposition \ref{prop-crepant-contr} and because 
$P^q$ is of index $4$. Then again the assertion follows by Lemma \ref{lemma-start}.

>From now on we assume that $q$ contracts $C^{\q}_0$,
i.e., $K_{X^{\q}}\cdot C^{\q}_0=0$.
In the case \ref{prop-m-1-2} by symmetry and by the above arguments 
we may assume that $q$ contracts $C^{\q}_1$.
Consider the decomposition 
\[
q \colon X^{\q} \stackrel{\varphi}{\longrightarrow} X^{\delta} 
\stackrel{\delta}{\longrightarrow} X, 
\]
where 
$\varphi$ contracts all the $K_{X^{\q}}$-trivial components of $C^{\q}$ except for 
$C^{\q}_0$. Put $C^\delta:=\varphi(C^{\q})$ and $C^\delta_0:=\varphi(C^{\q}_0)$.
Thus $-K_{X^\delta}$ is nef and big over $Z$ and $C^\delta_0$ is the only $K_{X^\delta}$-trivial 
curve on $X^\delta/Z$. 
Let $C^{\delta \delta}:=C^\delta- C^\delta_0$. 
Then $C^{\delta\delta}$ has at least 
two components. 
Let $P:=\delta(C_0^\delta)$ and $R^\delta= C^{\delta \delta}\cap C^\delta_0$.
By Proposition \ref{claim-non-Gorenstein} $R^\delta$ in not Gorenstein.

In the case \ref{prop-m-2}, 
$C^\delta_0$ contains at least three non-Gorenstein points:
$R^\delta$,
$P^\delta:=\varphi(P^{\q})$, and $Q^\delta:=\varphi(Q^{\q})$.
This contradicts Proposition \ref{prop-crepant-contr}.

In the case \ref{prop-m-1-1}, 
$P^\delta:=\varphi(P^{\q})$
is a locally imprimitive point of $(X^{\delta}, C^{\delta}_0)$.
By Proposition \ref{prop-crepant-contr} the singularity
$(X,P=\delta(C^\delta_0))$ is not of type $cA/*$.
If the index of $(X,P)$ is $\ge 4$, then $(X,P)$ is of type $cAx/4$
and we can apply Lemma \ref{lemma-start}.
Thus we assume that $(X,P)$ is of index $2$ and $n=2$.
Let $C_i\subset C$ be a component passing through $P$.
By \cite[Cor. 1.16]{Mori-1988} $(X,C_i)$ is primitive at $P$.
Further, $\xi:=f^*\eta=q_*\xi^{\q}$ is an $2$-torsion element 
of $\Clsc(X,C)$ and is not Cartier at $P$. 
This defines a double \'etale in codimension one cover $(X',C_i')\to (X,C_i)$
which does not splits over $P$.
Hence there is a point $Q\in (X,C_i)$ of even index.
This contradicts the classification \cite[Th. 2.2]{Kollar-Mori-1992}
(cf. \cite{Mori-2007}).

Consider the case \ref{prop-m-1-2}.
Then $P^\delta:=\varphi(P^{\q})$ is a point of index $\ge 4$
(because $\varphi$ is a crepant contraction).
Recall that $\varphi$  contracts $C^{\q}_1$ by our assumption..
Then by Proposition \ref{prop-crepant-contr} $(X^\delta,P^\delta)$
is a point of type $cAx/4$. As in the proof of 
Proposition \ref{prop-crepant-contr}, let 
$D\in |-K_{(X, \delta(P^\delta))}|$ be a general element
and let $S:=\delta^{-1}(D)$.
Then both $D$ and $S$ have only Du Val singularities 
and the contraction $\delta_S\colon S\to D$ is crepant.
Since $(S,P^\delta)$ is not of type $A_*$, 
the germ $(D,P)$ also cannot be of type $A_*$.
Hence, $(X,P)$ is not of type $cA/*$ and so
it is of type $cAx/4$ (because its index is $\ge 4$).
Then the assertion follows by Lemma \ref{lemma-start}.
\end{proof}

\begin{mparag}{Explicit forms.}
By Proposition \ref{prop-last} and Lemma \ref{lemma-start} $f\colon (X,C)\to (Z,o)$
is a quotient of an index-two $\QQ$-conic bundle $f'\colon (X',C')\to (Z',o')$ over a smooth base 
by $\muu_2$, where $\muu_2$ acts on $X'$ and $Z'$ freely in codimension one.
By \cite[Prop. 12.1.10]{Mori-Prokhorov-2006} there is a $\muu_2$-equivariant
diagram
\begin{equation*}
\label{eq-diag-last-2}
\xymatrix{X'\  \ar@{^{(}->}[r] \ar[rd]_{f}& \PP(1,1,1,2)\times \CC^2
\ar[d]^{p}
\\
&\CC^2}
\end{equation*}
where the actions of $\muu_2$ on $(\CC^2,0)\simeq (Z',o')$ and $\PP(1,1,1,2)$
are linear. Further, we can make coordinates $y_1,y_2,y_3, u,v$ in $\PP(1,1,1,2)$
and $\CC^2$ to be semi-invariant. By \cite[Th. 12.1]{Mori-Prokhorov-2006}
$X'$ is given by two semi-invariant equations
\begin{equation*}
\label{eq-eq-index2}
\left\{
\begin{array}{l}
q_1(y_1,y_2,y_3)-\psi_1(y_1,\dots,y_4;u,v)=0,
\\[7pt]
q_2(y_1,y_2,y_3)-\psi_2(y_1,\dots,y_4;u,v)=0,
\end{array}
\right.
\end{equation*}
where $\psi_i$ and $q_i$ are weighted quadratic in $y_1,\dots,y_4$
with respect to $\wt (y_1,\dots,y_4)=(1,1,1,2)$ and
$\psi_i(y_1,\dots,y_4;0,0)=0$. 
Since the action of $\muu_2$ on $Z\simeq \CC^2$ is free outside
$0$, this action is
given by $u\mapsto -u$, $v\mapsto -v$.
Modulo multiplication on $\pm1$ and permutations of 
$y_1,y_2,y_3$, we may assume also that $y_1\mapsto -y_1$, $y_2\mapsto -y_2$,
$y_3\mapsto y_3$.  Otherwise all the points of $\{y_4=0\}\cap C'$
are fixed by $\muu_2$, while $P$ is the only non-Gorenstein on $X$.

The central fiber $C'$ is defined by $q_1=q_2=0$. 
By Lemma \xref{lemma-start} $C'$ has exactly four components and 
$\muu_2$ does not fix any of them. 
Thus we may assume that $C'=\cup C_i'$, $i=1,2,3,4$ and 
$\muu_2$ interchanges $C_1'$ and $C_2'$ (resp. $C_3'$ and $C_4'$).
For any two components $C_i'\neq C_j'$
of $C'$, there is a linear form $l_{i,j}(y_1,\dots,y_3)$ that vanishes along 
$C_i'\cup C_j'$.  Then quadratic forms $l_{1,2}l_{3,4}$, 
$l_{1,3}l_{2,4}$, $l_{1,4}l_{2,3}$ vanish along $C'$
Hence they belong to the pencil $\lambda_1q_1+\lambda_2q_2$
and semi-invariant. This implies that the action of $\muu_2$ on 
the pencil is trivial. Moreover, we can put $q_1=l_{1,3}l_{2,4}$
and $q_2=l_{1,4}l_{2,3}$.
In view of the $\muu_2$-action we may assume 
that  $l_{1,3}=y_1+y_3$, $l_{2,4}=y_1-y_3$,
$l_{1,4}=y_2+y_3$, $ l_{2,3}=y_2-y_3$ 
after some linear coordinate change of $y_1$, $y_2$, $y_3$.

We claim
that $y_4\mapsto - y_4$. 
The arguments below are similar to ones in the proof
of \cite[Lemma 12.1.12]{Mori-Prokhorov-2006}.
Assume 
to the contrary that $y_4\mapsto y_4$.
Let $U\subset \PP(1,1,1,2)$ be the chart $y_4\neq 0$.
Then $U\simeq \CC^3_{z_1,z_2,z_3}/\muu_2(1,1,1)$.
Let $X^\sharp$ be the pull-back of $X\cap (U\times \CC^2_{u,v})$
on $\CC^3_{z_1,z_2,z_3}\times \CC^2_{u,v}$ and let $P^\sharp\in X^\sharp$
be the preimage of $P$.
Since the induced map $X^\sharp\to X$ is \'etale in codimension one,
$(X^\sharp, P^\sharp)\to (X,P)$ is the index-one cover.
Hence $(X^\sharp, P^\sharp)\to (X,P)/\muu_2$ is also the index-one cover
of
the terminal point
$(X,P)/\muu_2$ of index $4$
(the last is true because
the action of $\muu_2$ is free in codimension one).
Hence the morphism is a $\muu_4$-covering by the
structure of terminal  singularities.
However $(X,P)/\muu_2$ is the quotient of $(X^\sharp,P^\sharp)$ by
commuting $\muu_2$-actions:
\[
(z_1,z_2,z_3,u,v)\mapsto(-z_1,-z_2,-z_3,u,v),
(z_1,-z_2,z_3,-u,-v)
\]
This is a contradiction, and we have $y_4 \mapsto -y_4$ as
claimed. This finishes the proof of Theorem \ref{th-m}.
\end{mparag}


\begin{thebibliography}{Kaw96}

\bibitem[Kaw96]{Kawamata-1996}
Y. Kawamata.
\newblock Divisorial contractions to {$3$}-dimensional terminal quotient
  singularities.
\newblock In {\em Higher-dimensional complex varieties (Trento, 1994)}, pages
  241--246. de Gruyter, Berlin, 1996.

\bibitem[KM92]{Kollar-Mori-1992}
J. Koll{\'a}r and S. Mori.
\newblock Classification of three-dimensional flips.
\newblock {\em J. Amer. Math. Soc.}, 5(3):533--703, 1992.

\bibitem[Kol92]{Utah}
J.~Koll{\'a}r, editor.
\newblock {\em Flips and abundance for algebraic threefolds}.
\newblock Soci\'et\'e Math\'ematique de France, Paris, 1992.
\newblock Papers from the Second Summer Seminar on Algebraic Geometry held at
  the University of Utah, Salt Lake City, Utah, August 1991, Ast\'erisque No.
  211 (1992).

\bibitem[Kol99]{Kollar-1999-R}
J. Koll{\'a}r.
\newblock Real algebraic threefolds. {III}. {C}onic bundles.
\newblock {\em J. Math. Sci. (New York)}, 94(1):996--1020, 1999.
\newblock Algebraic geometry, 9.

\bibitem[Mor88]{Mori-1988}
S. Mori.
\newblock Flip theorem and the existence of minimal models for {$3$}-folds.
\newblock {\em J. Amer. Math. Soc.}, 1(1):117--253, 1988.

\bibitem[Mor02]{Mori-2002}
S. Mori.
\newblock On semistable extremal neighborhoods.
\newblock In {\em Higher dimensional birational geometry (Kyoto, 1997)},
  volume~35 of {\em Adv. Stud. Pure Math.}, pages 157--184. Math. Soc. Japan,
  Tokyo, 2002.

\bibitem[Mor07]{Mori-2007}
S. Mori.
\newblock Errata to \cite{Kollar-Mori-1992}.
\newblock {\em J. Amer. Math. Soc.}, 20(1):269--271, 2007.

\bibitem[MP06]{Mori-Prokhorov-2006}
S.~Mori and Yu. Prokhorov.
\newblock On $\mathbf {Q}$-conic bundles, 2006.

\bibitem[Pro97]{Prokhorov-1997_e}
Yu.~G. Prokhorov.
\newblock On the complementability of the canonical divisor for {M}ori
  fibrations on conics.
\newblock {\em Sbornik. Math.}, 188(11):1665--1685, 1997.

\bibitem[Rei87]{Reid-YPG1987}
M. Reid.
\newblock Young person's guide to canonical singularities.
\newblock In {\em Algebraic geometry, Bowdoin, 1985 (Brunswick, Maine, 1985)},
  volume~46 of {\em Proc. Sympos. Pure Math.}, pages 345--414. Amer. Math.
  Soc., Providence, RI, 1987.

\end{thebibliography}

\end{document}